\documentclass[11pt]{amsart}
\usepackage[foot]{amsaddr}
\usepackage{amsmath,amssymb,bm, graphicx,float}

\begin{document}
\newtheorem{theorem}{Theorem}[section]
\newtheorem{lemma}[theorem]{Lemma}
\newtheorem{corollary}[theorem]{Corollary}
\newtheorem{prop}[theorem]{Proposition}
\newtheorem{definition}[theorem]{Definition}
\newtheorem{remark}[theorem]{Remark}

 \def\ad#1{\begin{aligned}#1\end{aligned}}  \def\b#1{{\bf #1}} \def\hb#1{\hat{\bf #1}}
\def\a#1{\begin{align*}#1\end{align*}} \def\an#1{\begin{align}#1\end{align}}
\def\e#1{\begin{equation}#1\end{equation}} \def\t#1{\hbox{\rm{#1}}}
\def\dt#1{\left|\begin{matrix}#1\end{matrix}\right|}
\def\p#1{\begin{pmatrix}#1\end{pmatrix}} \def\c{\operatorname{curl}}
 \def\vc{\nabla\times } \numberwithin{equation}{section}
 \def\la{\circle*{0.25}}
\def\boxit#1{\vbox{\hrule height1pt \hbox{\vrule width1pt\kern1pt
     #1\kern1pt\vrule width1pt}\hrule height1pt }}
 \def\lab#1{\boxit{\small #1}\label{#1}}
  \def\mref#1{\boxit{\small #1}\ref{#1}}
 \def\meqref#1{\boxit{\small #1}\eqref{#1}}
\long\def\comment#1{}

\def\lab#1{\label{#1}} \def\mref#1{\ref{#1}} \def\meqref#1{\eqref{#1}}

\def\bg#1{{\pmb #1}} 

\title  [$C^1$ rectangular elements]
   {Rectangular $C^1$-$P_k$ finite elements with $Q_k$-bubble enrichment}

\author { Shangyou Zhang }
\address{Department of Mathematical  Sciences, University of Delaware, Newark, DE 19716, USA.}
\email{ szhang@udel.edu }

\date{}

\begin{abstract}
We enrich the $P_k$ polynomial space by $5$ ($k=4$), or $7$ ($k=5$), or 8 (all $k\ge 6$) $Q_k$ bubble functions to obtain a family of $C^1$-$P_k$ ($k\ge 4$)  finite elements on rectangular meshes. We show the uni-solvency,   the $C^1$-continuity and the quasi-optimal convergence. Numerical tests on the new $C^1$-$P_k$, $k=4,5,6,7$ and $8$, elements are performed.  
\end{abstract}

\vskip .3cm



\maketitle

\section{Introduction} 
In this work,  we construct $C^1$-$P_k$ ($k\ge 4$) finite elements
   by $Q_k$-bubble-enrichment on rectangular meshes for the following biharmonic equation, i.e.,  
   the plate bending equation,
\an{\label{bi} \ad{ \Delta^2 u & = f \quad \t{in } \ \Omega, \\
        u=\partial_{\b n} u & =0  \quad \t{on } \ \partial\Omega, } }
where $\Omega$ is a polygonal domain which can be subdivided into rectangles,
    and $\b n$ is the unit outer normal vector at the boundary.

Some famous finite elements were constructed in the early days,  for
  solving the biharmonic equation \eqref{bi}.
The $C^1$-$P_3$ Hsieh-Clough-Tocher element (1961,1965) was constructed in \cite{Ciarlet,Clough}.
The element is a macro-element where each base triangle is split into three by 
   connecting the bary-center to the three vertices.
The was extended to the family of $C^1$-$P_k$ ($k\ge 3$) 
   finite elements in \cite{Douglas}.

The $C^1$-$P_5$  Argyris element (1968) was constructed in \cite{Argyris}.
The $C^1$-$P_5$  Argyris element was extended to the family of 
   $C^1$-$P_k$ ($k\ge 5$) finite elements in \cite{Zen70,Zlamal}. 
The $C^1$-$P_5$  Argyris element was modified and extended to the family of  
   $C^1$-$P_k$ ($k\ge 5$) full-space finite elements in \cite{Morgan-Scott}.
The $C^1$-$P_5$  Argyris element was also extended to 3D $C^1$-$P_k$ ($k\ge 9$)
   elements on tetrahedral meshes in \cite{Zenisek,Z3d,Z4d}.
 
The $C^1$-$P_4$ Bell element (1969) was constructed in \cite{Bell}.
The Bell element eliminates all degrees of freedom at edges by
  limiting the polynomial degree of the normal derivative.
The $C^1$-$P_4$ Bell element was extended to three families of 
    $C^1$-$P_{2m+1}$ ($m\ge 3$) finite elements in \cite{Xu-Zhang7,Xu-Zhang}. 
As the Bell finite elements do not have any degrees of freedom on edges,
   the polynomial degree above must be an odd one.

 The $C^1$-$P_3$ \ Fraeijs de Veubeke-Sander element (1964,1965)
   was constructed in \cite{Fraeijs,Fraeijs68,Sander},
  where each base quadrilateral is split into 4 sub-triangles by the
    two diagonal lines, on quadrilateral meshes.
The $C^1$-$P_3$ Fraeijs de Veubeke-Sander element is extended to two families
  of $C^1$-$P_k$ ($k\ge 3$) finite elements in \cite{Zhang-F}.

\begin{figure}[H]\centering
 \begin{picture}(300,260)(0,-10) 
 \def\c{\circle*{4}}\def\h{\vector(1,0){13}}\def\v{\vector(0,1){13}}
 \def\d{\multiput(0.5,-0.5)(-1,1){2}{\vector(1,1){11}}}  

 \put(0,135){\begin{picture}(100,108)(0,0)   
  \put(0,0){\line(1,0){100}}  \put(0,0){\line(0,1){100}}   
  \put(100,100){\line(-1,0){100}}  \put(100,100){\line(0,-1){100}}  
  \multiput(0,0)(50,0){3}{\multiput(0,0)(0,50){3}{\c}}  
  \multiput(2,2)(50,0){3}{\v}  \multiput(2,102)(50,0){3}{\v}  
  \multiput(2,2)(0,50){3}{\h}  \multiput(102,2)(0,50){3}{\h} 
    \multiput(2,2)(0,100){2}{\d} \multiput(102,2)(0,100){2}{\d} 
  
     % \put(-15,-12){$\b x_1$} \put(103,-12){$\b x_2$}
     % \put(-15,104){$\b x_4$} \put(88,104){$\b x_3$}
 \end{picture} }
  
 \put(0,-5){\begin{picture}(100,108)(0,0)   
  \put(0,0){\line(1,0){100}}  \put(0,0){\line(0,1){100}}   
  \put(100,100){\line(-1,0){100}}  \put(100,100){\line(0,-1){100}}  
  \multiput(0,0)(50,0){3}{\multiput(0,0)(0,50){3}{\c}}  
  \multiput(2,2)(100,0){2}{\v}  \multiput(2,102)(100,0){2}{\v}  
  \multiput(2,2)(0,100){2}{\h}  \multiput(102,2)(0,100){2}{\h} 
    \multiput(2,2)(0,100){2}{\d} \multiput(102,2)(0,100){2}{\d} 
  
     % \put(-15,-12){$\b x_1$} \put(103,-12){$\b x_2$}
     % \put(-15,104){$\b x_4$} \put(88,104){$\b x_3$}
 \end{picture} }
  
 \put(160,-5){\begin{picture}(100,108)(0,0)   
  \put(0,0){\line(1,0){100}}  \put(0,0){\line(0,1){100}}   
  \put(100,100){\line(-1,0){100}}  \put(100,100){\line(0,-1){100}}  

  \multiput(0,0)(100,0){2}{\multiput(0,0)(0,50){3}{\c}} 
     \multiput(50,0)(100,0){1}{\multiput(0,0)(0,100){2}{\c}}  

  \multiput(2,2)(100,0){2}{\v}  \multiput(2,102)(100,0){2}{\v}  
  \multiput(2,2)(0,100){2}{\h}  \multiput(102,2)(0,100){2}{\h} 
    \multiput(2,2)(0,100){2}{\d} \multiput(102,2)(0,100){2}{\d} 
  
     % \put(-15,-12){$\b x_1$} \put(103,-12){$\b x_2$}
     % \put(-15,104){$\b x_4$} \put(88,104){$\b x_3$}
 \end{picture} }
  
 \put(160,135){\begin{picture}(100,108)(0,0)   
  \put(0,0){\line(1,0){100}}  \put(0,0){\line(0,1){100}}   
  \put(100,100){\line(-1,0){100}}  \put(100,100){\line(0,-1){100}}  
  \multiput(0,0)(100,0){2}{\multiput(0,0)(0,50){3}{\c}} 
     \multiput(50,0)(100,0){1}{\multiput(0,0)(0,100){2}{\c}}  
  \multiput(2,2)(50,0){3}{\v}  \multiput(2,102)(50,0){3}{\v}  
  \multiput(2,2)(0,50){3}{\h}  \multiput(102,2)(0,50){3}{\h} 
    \multiput(2,2)(0,100){2}{\d} \multiput(102,2)(0,100){2}{\d} 
   
 \end{picture} }
  
 \end{picture}
 \caption{Top-left: The 25 degrees of freedom 
     for the $C^1$-$Q_4$ BFS element; \ 
     Top-right: The 24 degrees of freedom for the $C^1$-$Q_4$
       serendipity finite element; \
    Bottom-left: The 21 degrees of freedom for the $C^1$-$Q_4$
       Bell element; \ 
     Bottom-right: The 20 degrees of freedom for the new $C^1$-$P_4$
        finite element. } \label{T-4}
 \end{figure}
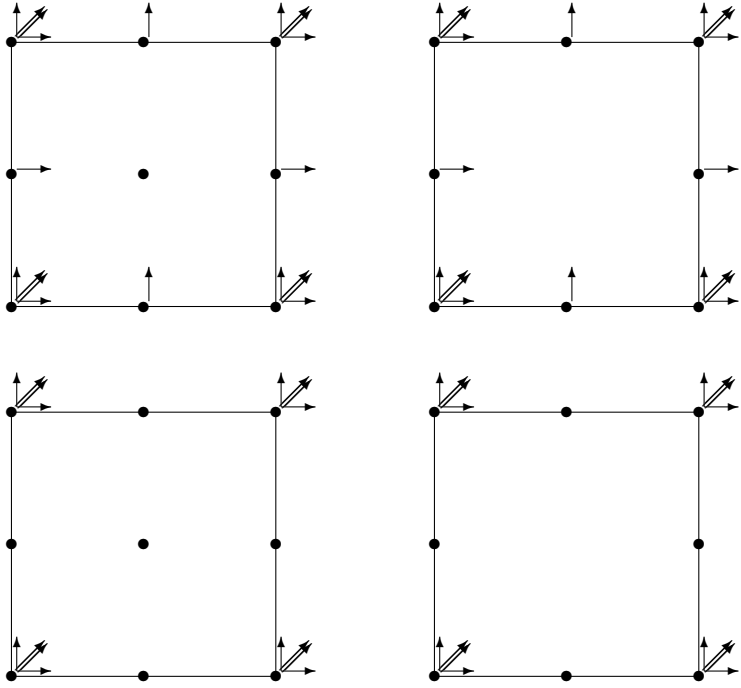

The $C^1$-$Q_3$ Bogner-Fox-Schmit element (1965) was constructed in \cite{Bogner}.
The $C^1$-$Q_3$ BFS element was extended to three families of $C^1$-$Q_k$ ($k\ge 3$)
  finite elements on rectangular meshes in \cite{Zhang-C1Q}.
The $C^1$-$Q_k$ Bell elements were constructed in \cite{Hu-Hongling}, where the polynomial
  degree of the normal derivative is reduced.  
The $C^1$-$Q_k$ serendipity elements were constructed in \cite{Zhang-seren},
  where all redundant internal degrees of freedom of the dofs of $C^1$-$Q_k$ 
    are eliminated and replaced by $P_{k-8}$ internal Lagrange nodes.
In this work, we use some such $C^1$-$Q_k$ bubbles to enrich the $P_k$ space in
  the $C^1$-$P_k$ finite element construction.

The $C^1$-$Q_4$ BFS element has 25 degrees of freedom (shown in Figure \ref{T-4}) on
  each square.
The serendipity element eliminates the 1 internal dof of the $Q_4$ BFS' 25 dofs
   and has 24 dofs each element.
The Bell element eliminate an edge-derivative dof of the $Q_4$ BFS' dofs and has 21 dofs
   per element.
The newly constructed $C^1$-$Q_4$ element eliminates both eliminated dofs (1 plus 4)
  above has 20 dofs each element.

In this work, we enrich the $P_k$ polynomial space by $5$ ($k=4$), or $7$ ($k=5$),
  or 8 (all $k\ge 6$) $Q_k$ bubble functions to obtain a family of $C^1$-$P_k$ ($k\ge 4$) 
  finite elements on rectangular meshes.
We show the uni-solvency,   the $C^1$-continuity and the quasi-optimal convergence.
Numerical tests on the new $C^1$-$P_k$, $k=4,5,6,7$ and $8$,
    elements are performed,  confirming the theory.
They are compared with the $C^1$-$Q_k$ BSF counterparts.

\section{The bubble-enriched $C^1$-$P_4$ finite element}

Let $\mathcal Q_h=\{ T \}$ be a uniform square mesh on the domain $\Omega$. 
On a square (or a rectangle) $T$,  the $C^1$-$Q_k$ Bell element, 
   a sub-element of the Bogner-Fox-Schmit (BFS) finite element, is defined by, cf. \cite{Hu-Hongling},
    for $k\ge 4$,  
\an{\label{W-T} W_k(T) =\{ v \in Q_k(T) : \partial_{\b n} v|_e \in Q_{k-1}(e), \ e \in \partial T\}, }
where $\partial_{\b n}$ denotes a normal derivative on the edge $e$,
   and $Q_k=\t{span} \{ x^{ k_1}y^{k_2} : 0\le k_1, k_2\le k \}$.  
For the finite element $V_T$,  
    the degrees of freedom of the Bell element are defined by, cf.
   Figure \ref{f-dof2}, 
\an{\label{dof2} F_m(v) = \begin{cases}
     v ,  & \t{at } \ \b x_1+\frac h{k-2}\langle i, j\rangle, \; i,j=0,\dots,k-2, \\
     \partial_x v, 
          & \t{at } \ \b x_1+ h \langle i,\frac j{k-3}\rangle, \; i=0,1, \ j=0,\dots,k-3, \\ 
     \partial_y v, 
          & \t{at } \ \b x_1+ h \langle \frac i{k-3},j\rangle , \; i=0,\dots, k-3,\ j=0,1, \\
     \partial_{xy} v, 
          & \t{at } \ \b x_1+h\langle i,j\rangle, \; i,j=0, 1, \end{cases} }
where $h$ is the $x$-size and the $y$-size of the square $T$.

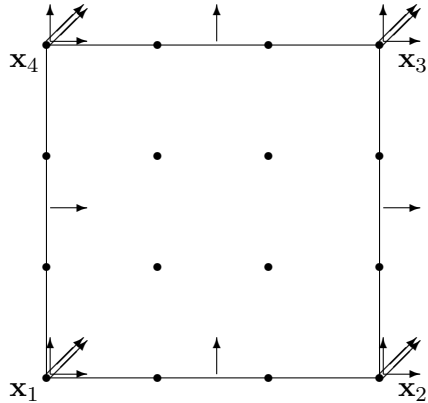
\begin{figure}[H] \centering \setlength\unitlength{1.4pt}
\begin{picture}(110,110)(0,0)
\put(10,5){\begin{picture}(110,110)(0,0) 
  \put(-10,-5){$\b x_1$}   \put(95,-5){$\b x_2$} 
  \put(-10,84){$\b x_4$}   \put(95,84){$\b x_3$} 
  \multiput(0,0)(90,0){2}{\line(0,1){90}}\multiput(0,0)(0,90){2}{\line(1,0){90}}
  \multiput(0,0)(30,0){4}{\multiput(0,0)( 0,30){4}{\circle*{2}}}
  \multiput(1,1)(90,0){2}{\multiput(0,0)( 0,45){3}{\vector(1,0){10}}}
  \multiput(1,1)(0,90){2}{\multiput(0,0)(45, 0){3}{\vector(0,1){10}}}
  \multiput(1,0)(90,0){2}{\multiput(0,0)( 0,90){2}{\vector(1,1){10}}}
  \multiput(0,1)(90,0){2}{\multiput(0,0)( 0,90){2}{\vector(1,1){10}}}\end{picture} } 
   \end{picture}
   \caption{\label{f-dof2} The degrees of freedom of the $C^1$-$Q_5$ Bell finite element,
       cf. \eqref{dof2}.
         }
   \end{figure}

The finite element nodal basis functions, dual to the degrees of freedom \eqref{dof2},
  are denoted by
\an{\label{b-basis} \begin{cases} 
     b_{1}^{i,j},  &i,j=0,\dots,k-2, \\
     b_{2}^{i,j}, 
          & i=0,1, \ j=0,\dots,k-3, \\ 
     b_{3}^{i,j}& i=0,\dots, k-3,\ j=0,1, \\
     b_{4}^{i,j}  & i,j=0, 1. \end{cases} }

For $k=4$, to be $C^1$ and to include $P_k$ space on each edge,  we need 
   at least $4(3+2)=20$ degrees of freedom.
While $\dim P_4=15$,  we select 5 Bell-bubble basis functions 
  $\{  b_{1}^{1,0}, b_{1}^{2,0},  b_{2}^{1,0}, b_{3}^{1,0}, b_{4}^{1,0} \}$ of $W_4$ in \eqref{W-T} from 
  \eqref{b-basis},   as shown in Figure \ref{T}.
Enriched by the 5 bubble functions,  we define the $C^1$-$P_4$ finite element by
\an{\label{V-4} V_4(T)=\t{span} \{ P_4(T), \ 
       b_{1}^{1,0}, b_{1}^{2,0},  b_{2}^{1,0}, b_{3}^{1,0}, b_{4}^{1,0} \}.  }
We define the following degrees of freedom for the space $V_4(T)$, ensuring the global 
   $C^1$ continuity, \  by $F_m(p)=$
\an{\label{d-4} &\begin{cases} p(\b x_i), \partial_x p(\b x_i), \partial_y p(\b x_i),
        \partial_{x y} p(\b x_i), \quad \ i=1,2,3,4, \\
    p(\frac { \b x_1+\b x_{2}}{2}),\; 
    p(\frac { \b x_2+\b x_{3}}{2}), \; 
     p(\frac { \b x_4+\b x_{3}}{2}), \; 
    p(\frac { \b x_1+\b x_{4}}{2}).
    \end{cases} }

\begin{lemma} The degrees of freedom \eqref{d-4} uniquely determine the 
   $V_4(T)$ functions in \eqref{V-4}.
\end{lemma}

\begin{proof}We count the dimension of $V_4$ in \eqref{V-4} and the number $N_{\text{dof}}$ of
   degrees of freedom in \eqref{d-4},
\a{ \dim V_4(T) &= \dim P_4 + 5 =15+5=20, \\
    N_{\text{dof}} &= 4\cdot 4 +4 =20. }
Thus the uni-solvency is determined by uniqueness.

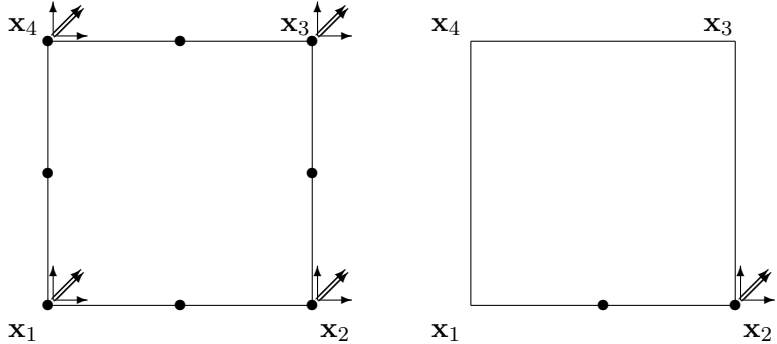
\begin{figure}[H]\centering
 \begin{picture}(280,140)(0,-10) 
 \def\c{\circle*{4}}\def\h{\vector(1,0){13}}\def\v{\vector(0,1){13}}
 \def\d{\multiput(0.5,-0.5)(-1,1){2}{\vector(1,1){11}}}  

 \put(0,0){\begin{picture}(100,108)(0,0)   
  \put(0,0){\line(1,0){100}}  \put(0,0){\line(0,1){100}}   
  \put(100,100){\line(-1,0){100}}  \put(100,100){\line(0,-1){100}}  
  \multiput(0,0)(50,0){3}{\c} \multiput(0,100)(50,0){3}{\c}  
   \multiput(0,50)(100,0){2}{\c}  
  \multiput(2,2)(100,0){2}{\v}  \multiput(2,102)(100,0){2}{\v}  
  \multiput(2,2)(0,100){2}{\h}  \multiput(102,2)(0,100){2}{\h} 
    \multiput(2,2)(0,100){2}{\d} \multiput(102,2)(0,100){2}{\d} 
  
      \put(-15,-12){$\b x_1$} \put(103,-12){$\b x_2$}
      \put(-15,104){$\b x_4$} \put(88,104){$\b x_3$}
 \end{picture} }
 
 \put(160,0){\begin{picture}(100,108)(0,0)   
  \put(0,0){\line(1,0){100}}  \put(0,0){\line(0,1){100}}   
  \put(100,100){\line(-1,0){100}}  \put(100,100){\line(0,-1){100}}  \put(102,2){\v} 
     \put(102,2){\h} \put(102,2){\d} 
   \put(50,0){\c}   \put(100,0){\c}   
      \put(-15,-12){$\b x_1$} \put(103,-12){$\b x_2$}
      \put(-15,104){$\b x_4$} \put(88,104){$\b x_3$}
 \end{picture} }

 \end{picture}
 \caption{The 20 degrees of freedom 
     for the  bubble-enriched $C^1$-$P_4$ element in \eqref{d-4},
     and the 5 bubble functions $\{  b_{1}^{1,0}, b_{1}^{2,0},  b_{2}^{1,0}, b_{3}^{1,0}, b_{4}^{1,0} \}$
     from \eqref{b-basis} used to define the bubble-enriched 
     $C^1$-$P_4$ finite element in \eqref{V-4}. } \label{T}
 \end{figure}

Let $p\in V_4(T)$ in \eqref{V-4} and $F_m(p)=0$ for all degrees of freedom in \eqref{d-4}.
Let \an{\label{p-4} p=p_4+\sum_{\ell=1}^5 c_\ell b_{\ell_1}^{\ell_2,\ell_3}
       \quad \ \t{for some } \ p_4\in P_4(T),  }
       where $b_{\ell_1}^{\ell_2,\ell_3}$ are defined in \eqref{V-4}.
As all $b_{\ell_1}^{\ell_2,\ell_3}$ vanish at these points, we have
\an{\label{p-4-1}\ad{
   && p_4(\b x_1)&=0, \ & \partial_y p_4(\b x_1)&=0, \ & p_4(\frac{\b x_1+\b x_4}2)&=0, & \\
    &&  p_4(\b x_4)& =0, \ & \partial_y  p_4(\b x_4)&=0,  } }
and consequently $p_4|_{\b x_1\b x_4}=0$ as the degree 4 polynomial has 5 zero points.
Thus \a{ p_4=\lambda_{14} p_3 \quad \ \t{for some } \ p_3\in P_3(T), }
where $\lambda_{14}$ is a linear polynomial vanishing at the line $\b x_1\b x_4$
  and assuming value $1$ at $\b x_2$.
Now, as all $b_{\ell_1}^{\ell_2,\ell_3}$ have these vanishing degrees of freedom,  we have
\a{ \partial_x p_4(\b x_1)&= h p_3(\b x_1)=0, \\
    \partial_{x y} p_4(\b x_1)&= h \partial_y p_3(\b x_1)=0, \\
    \partial_x p_4(\b x_4)&= h  p_3(\b x_4)=0,  \\
    \partial_{x y} p_4(\b x_4)&= h \partial_y p_3(\b x_4)=0,  } and consequently $p_3|_{\b x_1\b x_4}=0$.
    
We can then factor out another linear polynomial that
\an{\label{p-4-2} p_4= \lambda_{14}^2 p_2 \quad \ \t{for some } \ p_2\in P_2(T). }
As $b_{\ell_1}^{\ell_2,\ell_3}$ have these three degrees of freedom vanished,  
  we then have
  \a{ p_4(\frac{\b x_4+\b x_3}2)&= \frac 1{2^2} \cdot p_2(\frac{\b x_4+\b x_3}2)=0, \\
    p_4( \b x_3) &=1 \cdot p_2( \b x_3)=0, \\
    \partial_x p_4( \b x_3) &= \frac 1{h^2} \cdot p_2( \b x_3)+1 \cdot \partial_x  p_2( \b x_3)=0, }
and consequently $p_2|_{\b x_4\b x_3}=0$.
We factor out this linear polynomial factor as
\a{ p_4= \lambda_{14}^2 \lambda_{43} p_1  \quad \ \t{for some } \ p_1\in P_1(T), }
where $\lambda_{43}$ is a linear polynomial vanishing at the line $\b x_4\b x_3$
  and assuming value $1$ at $\b x_1$.

As $b_{\ell_1}^{\ell_2,\ell_3}$ again have the following two degrees of freedom vanished,  
  we then have
\a{ \partial_y p_1( \b x_4)&=1\cdot\frac{-1}h \cdot p_1( \b x_4 )=0, \\
    \partial_{x y} p_1( \b x_4)&= \partial_x p_2( \b x_3)=0,   }
and consequently $p_1|_{\b x_3\b x_4}=0$.
We factor out this last linear polynomial factor as
\a{ p_4= \lambda_{14}^2  \lambda_{43}^2 c \quad \ \t{for some } \ c\in P_0(T). } 
Evaluating the last degree of freedom, cf. Figure \ref{T},   we have
\a{  p_4(\frac{\b x_2+\b x_3}2)&=1\cdot \frac 1{2^2}\cdot c =0. }
Thus $c=0$ and $p_4=0$ in \eqref{p-4}.

As $p_4=0$, evaluating $p$ in \eqref{p-4} sequentially at the degrees of freedom of 
$b_{\ell_1}^{\ell_2,\ell_3}$, it follows that
  \a{ c_1=\dots=c_5 = 0.  }
The lemma is proved as $p=0$ in \eqref{p-4}.
\end{proof}

\section{The bubble-enriched $C^1$-$P_5$ finite element}

Enriched by the following seven bubble functions,  we define the bubble-enriched
    $C^1$-$P_5$ finite element by
\an{\label{V-5} V_5(T)=\t{span} \{ P_5(T), \ b_1^{1,0}, b_1^{2,0}, b_1^{3,0}, b_{3}^{1,0}, 
    b_{2}^{2,0}, b_{3}^{2,0}, b_{4}^{1,0}\},  }
where $b_\ell^{i_j}$ is a basis function in \eqref{b-basis},
    dual to the degrees of freedom in \eqref{dof2}.
We define the following degrees of freedom for the space $V_5(T)$, ensuring the global 
   $C^1$ continuity, \  by $F_m(p)=$
\an{\label{d-5} &\begin{cases} p(\b x_i), \partial_x p(\b x_i), \partial_y p(\b x_i),
        \partial_{x y} p(\b x_i),   \ & i=1,2,3,4, \\
    p(\frac { j\b x_1+(3-j)\b x_{2}}3),\; p(\frac { j\b x_2+(3-j)\b x_{3}}3, \;
     & j=1,\dots, k-3, \\ 
     p(\frac { j\b x_4+(3-j)\b x_{3}}3), \;p(\frac { j\b x_1+(3-j)\b x_{4}}3), \;
     & j=1,\dots, 2, \\ 
    \partial_{y} p(\frac { \b x_4+ \b x_{3}}2),\;
    \partial_{y} p(\frac { \b x_1+ \b x_{2}}2),\;  \\  
    \partial_{x} p(\frac { \b x_2+ \b x_{3}}2),
    \partial_{x} p(\frac { \b x_1+ \b x_{4}}2).
    \end{cases} } 

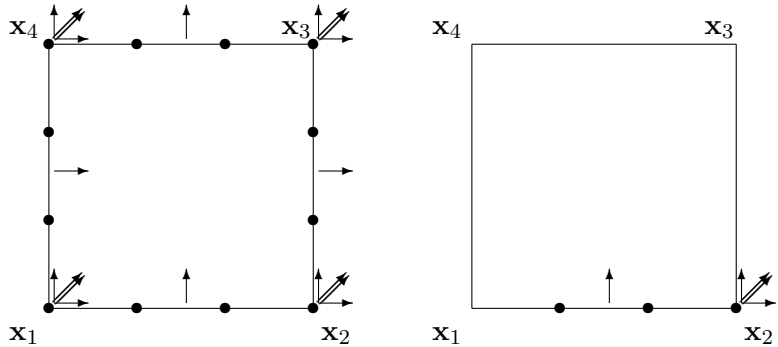
\begin{figure}[H]\centering
 \begin{picture}(280,140)(0,-10) 
 \def\c{\circle*{4}}\def\h{\vector(1,0){13}}\def\v{\vector(0,1){13}}
 \def\d{\multiput(0.5,-0.5)(-1,1){2}{\vector(1,1){11}}}  

 \put(0,0){\begin{picture}(100,108)(0,0)   
  \put(0,0){\line(1,0){100}}  \put(0,0){\line(0,1){100}}   
  \put(100,100){\line(-1,0){100}}  \put(100,100){\line(0,-1){100}}  
  \multiput(0,0)(33.33,0){4}{\c}  \multiput(0,100)(33.33,0){4}{\c} 
  \multiput(0,33.33)(100,0){2}{\c} \multiput(0,66.66)(100,0){2}{\c} 
  
  \multiput(2,2)(50,0){3}{\v}  \multiput(2,102)(50,0){3}{\v}  
  \multiput(2,2)(0,50){3}{\h}  \multiput(102,2)(0,50){3}{\h} 
    \multiput(2,2)(0,100){2}{\d} \multiput(102,2)(0,100){2}{\d} 
  
      \put(-15,-12){$\b x_1$} \put(103,-12){$\b x_2$}
      \put(-15,104){$\b x_4$} \put(88,104){$\b x_3$}
 \end{picture} }
 
 \put(160,0){\begin{picture}(100,108)(0,0)   
  \put(0,0){\line(1,0){100}}  \put(0,0){\line(0,1){100}}   
  \put(100,100){\line(-1,0){100}}  \put(100,100){\line(0,-1){100}}  \put(102,2){\v} 
     \put(102,2){\h} \put(102,2){\d} 
   \put(33.33,0){\c} \put(66.66,0){\c} \put(52,2){\v}   \put(100,0){\c}   
      \put(-15,-12){$\b x_1$} \put(103,-12){$\b x_2$}
      \put(-15,104){$\b x_4$} \put(88,104){$\b x_3$}
 \end{picture} }

 \end{picture}
 \caption{The $28$ degrees of freedom 
     for the enriched $C^1$-$P_5$ finite element in \eqref{V-5},
     and the 7 bubble functions $\{ b_1^{1,0}, b_1^{2,0}, b_1^{3,0}, b_{3}^{1,0}, 
    b_{2}^{2,0}, b_{3}^{2,0}, b_{4}^{1,0} \}$
      used to define \eqref{V-5}. } \label{T5}
 \end{figure}

\begin{lemma} The degrees of freedom \eqref{d-5} uniquely determine the 
   $V_5(T)$ functions in \eqref{V-5}.
\end{lemma}

\begin{proof} We count the dimension of $V_5$ in \eqref{V-5} and the number $N_{\text{dof}}$ of
   degrees of freedom in \eqref{d-5},
\a{ \dim V_5(T) &= \dim P_5 + 7 =21+7=28, \\
    N_{\text{dof}} &= 16+4\cdot 3=28. }
Thus the uni-solvency is determined by uniqueness.

Let $p\in V_5(T)$ in \eqref{V-5} and $F_m(p)=0$ for all degrees of freedom in \eqref{d-5}.
Let \an{\label{p-5} p=p_5+\sum_{ \ell=1}^{7} c_\ell b_{\ell_1}^{\ell_2,\ell_3}
     \quad \ \t{for some } \ p_5\in P_5(T).  }
Repeating \eqref{p-4-1} and \eqref{p-4-2},  we have 
\an{\label{p-5-1} p_5= \lambda_{14}^2 p_3 \quad \ \t{for some } \ p_3\in P_3(T). }

As $b_{\ell_1}^{\ell_2,\ell_3}$ have these four degrees of freedom vanished,  
  we then have 
\a{ p_5(\frac{ 2 \b x_4+\b x_3} 3 )&= \frac{2^2}{3^2} \cdot p_3(\frac{ 2 \b x_4+\b x_3} 3 )=0, \\
    p_5(\frac{  \b x_4+2\b x_3} 3 )&= \frac{1^2}{3^2} \cdot p_3(\frac{ \b x_4+2 \b x_3} 3 )=0, \\
    p_5( \b x_3 )&= 1 \cdot p_3( \b x_3  )=0, \\
  \partial_x p_5( \b x_3 )&= \frac{-2}{h} \cdot p_3( \b x_3  )+  \partial_x p_3( \b x_3  )=0, }
   and consequently $p_3|_{\b x_4\b x_3}=0$.

We factor out this linear polynomial factor as
\a{ p_5= \lambda_{14}^2 \lambda_{43} p_2  \quad \ \t{for some } \ p_2\in P_2(T). }
Evaluating the following three degrees of freedom, we have
\a{ \partial_y p_5(\frac{ \b x_4+\b x_3}2)&= \frac 1{2^2} \cdot 
                     \frac{1} h p_2(\frac{ \b x_4+\b x_3}3)=0,\\
        \partial_y p_5( \b x_3 )&=1 \cdot \frac{1} h p_2( \b x_3)=0,\\ 
    \partial_{x y} p_5( \b x_3 )&=\frac{-2}h \cdot+  \frac{-1} h  p_2(\b x_3) +
             1 \cdot \frac{-1} h \partial_{x}p_2(\b x_3)=0,
  } and consequently $p_2|_{\b x_4\b x_3}=0$.
We factor out this linear polynomial as
\an{\label{p-5-2} p_5= \lambda_{14}^2 \lambda_{43}^2 p_1 \quad \ \t{for some } \ p_1\in P_1(T). }

We evaluate the function values in the middle of edge $\b x_2\b x_3$, cf. Figure \ref{T5},
\a{ p_5(\frac{2\b x_2+\b x_3}3)&= 1^2 \cdot \frac{2^2}{3^2}\cdot p_1(\frac{2\b x_2+\b x_3}3)=0,\\
    p_5(\frac{\b x_2+2\b x_3}3)&= 1^2 \cdot \frac{1^2}{3^2}\cdot p_1(\frac{\b x_2+2\b x_3}3)=0. }
Thus $p_1$ vanishes on the edge and we have
\a{ p_5= \lambda_{14}^2 \lambda_{43}^2 \lambda_{23} p_0 \quad \ \t{for some } \ p_0\in P_0(T). }
Evaluating the last degree of freedom, cf. Figure \ref{T5},
\a{ \partial_x p_5(\frac{\b x_2 +\b x_3}2)=1 \cdot \frac 1 {3^2} \cdot \frac {1} h p_0 =0. }
Thus, $p_0=0$ and consequently $p_5=0$ in \eqref{p-5}.

Evaluating $p$ in \eqref{p-5} sequentially at the degrees of freedom of $b_{\ell_1}^{\ell_2,\ell_3}$,
  it follows that
  \a{ c_1=\dots=c_{7} = 0, \quad \t{and }\ p=0.  }
The lemma is proved.
\end{proof}

\section{The  bubble-enriched $C^1$-$P_k$ ($k\ge 6$) finite element}

For all $k\ge 6$, we enrich the $P_k$ space by following 8 bubbles to define the $C^1$-$P_k$  
  finite element, cf. Figure \ref{T-k},
\an{\label{V-k} V_k(T)=\t{span} \{ P_k(T), b_{1}^{1,0}, b_{1}^{2,0}, b_{3}^{1,0}, b_{3}^{2,0},
   b_{1}^{k-2,0},  b_{2}^{1,0},b_{3}^{k-3,0}, b_{4}^{1,0}  \},  }
where $b_\ell^{i,j}$ is a basis function in \eqref{b-basis} dual to a degree of freedom in
   \eqref{dof2}.
We define the following degrees of freedom for the space $V_k(T)$, which also ensure the global 
   $C^1$ continuity, \ cf. Figure \ref{T-k}, by $F_m(p)=$
\an{\label{d-k} &\begin{cases} p(\b x_i), \partial_x p(\b x_i), \partial_y p(\b x_i),
        \partial_{x y} p(\b x_i),   \ & i=1,2,3,4, \\
    p(\frac { j\b x_1+(k-2-j)\b x_{2}}{k-2}),\; p(\frac { j\b x_2+(k-2-j)\b x_{3}}{k-2}), \;
     & j=1,\dots, k-3, \\ 
    p(\frac { j\b x_4+(k-2-j)\b x_{3}}{k-2}), \;   p(\frac { j\b x_1+(k-2-j)\b x_{4}}{k-2}), \;
     & j=1,\dots, k-3, \\  
    \partial_{y} p(\frac { j\b x_4+(k-3-j)\b x_{3}}{k-3}),\;
    \partial_{y} p(\frac { j\b x_1+(k-3-j)\b x_{2}}{k-3}), & j=1,\dots, k-4, \\  
    \partial_{x} p(\frac { j\b x_2+(k-3-j)\b x_{3}}{k-3}),\;
    \partial_{x} p(\frac { j\b x_1+(k-3-j)\b x_{4}}{k-3}), & j=1,\dots, k-4,\\
     p(\frac {  i\b x_2+j\b x_{4}+(k-4-i-j)\b x_1 }{k-2}), 
            & i=1,\dots,k-7,\\ & j=1,\dots, i, \;k>7. 
    \end{cases} }

\begin{lemma} The degrees of freedom \eqref{d-k} uniquely determine the 
   $V_k(T)$ functions in \eqref{V-k}.
\end{lemma}

\begin{proof} We count the dimension of $V_k$ in \eqref{V-k} and the number $N_{\text{dof}}$ of
   degrees of freedom in \eqref{d-k},
\a{ \dim V_k(T) &= \dim P_k + 8 =\frac{(k+1)(k+2)}2+8  \\
                &=\begin{cases} 36, & k=6, \\
                    44, & k=7, \\
                    \frac 12 k^2+\frac 32 k + 9, \qquad & k\ge 8, \end{cases}\\
    N_{\text{dof}} &=16+4(2k-7)+\frac{(k-7)(k-6)}2 \\
                &=\begin{cases} 40, & k=6, \\
                    48, & k=7, \\
                    \frac 12 k^2+\frac 32 k + 9, \qquad & k\ge 8. \end{cases} }
Thus, the uni-solvency is determined by uniqueness.

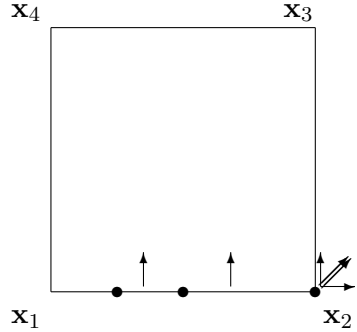
\begin{figure}[H]\centering
 \begin{picture}(140,140)(0,-10) 
 \def\c{\circle*{4}}\def\h{\vector(1,0){13}}\def\v{\vector(0,1){13}}
 \def\d{\multiput(0.5,-0.5)(-1,1){2}{\vector(1,1){11}}}

 \put(0,0){\begin{picture}(100,108)(0,0)   
  \put(0,0){\line(1,0){100}}  \put(0,0){\line(0,1){100}}   
  \put(100,100){\line(-1,0){100}}  \put(100,100){\line(0,-1){100}}  \put(102,2){\v} 
    \put(102,2){\h} \put(102,2){\d} 
    \put(100,0){\c}
  
    \put(35,2){\v}\put(68,2){\v}
    \put(25,0){\c} \put(50,0){\c} 
   
      \put(-15,-12){$\b x_1$} \put(103,-12){$\b x_2$}
      \put(-15,104){$\b x_4$} \put(88,104){$\b x_3$}
 \end{picture} }

 \end{picture}
 \caption{The 8 bubble functions $\{b_{1}^{1,0}, b_{1}^{2,0}, b_{3}^{1,0}, b_{3}^{2,0},
   b_{1}^{k-2,0},$ $  b_{2}^{1,0},b_{3}^{k-3,0}, b_{4}^{1,0} \}$
      used to define  the $C^1$-$P_k$ ($k\ge 6$) finite element in \eqref{V-k}. } \label{T-k}
 \end{figure}

Let $p\in V_k(T)$ in \eqref{V-k} and $F_m(p)=0$ for all degrees of freedom in \eqref{d-k}.
Let \an{\label{p-k} p=p_k+\sum_{ \ell=1}^{8} c_\ell b_{\ell_1}^{\ell_2,\ell_3}
   \quad \ \t{for some } \ p_k\in P_k(T).  }
Though we have one more dof and one more polynomial coefficient each step, 
  repeating \eqref{p-5-1} and \eqref{p-5-2},  we get 
\a{ p_k= \lambda_{14}^2\lambda_{43}^2 p_{k-4} \quad \ \t{for some } \ p_{k-4}\in P_{k-4}(T). }

As $b_{\ell_1}^{\ell_2,\ell_3}$ have the following degrees of freedom vanished,   we   have
\a{    &\quad \ p_{k}(\frac{j\b x_2+(k-2-j)\b x_3}{k-2})\\
        & = 1\cdot\frac{j^2}{(k-2)^2}
      p_{k-4}(\frac{j\b x_2+(k-2-j)\b x_3}{k-2})\\
        & =0, \; \qquad
   j=1,\dots,k-3,  }
and consequently $p_{k-4}|_{\b x_2\b x_3}=0$.
We factor out this linear polynomial factor as
\a{ p_{k}= \lambda_{14}^2 \lambda_{43}^2 \lambda_{23} p_{k-5}  
     \quad \ \t{for some } \ p_{k-5}\in P_{k-5}(T). } 
As $b_{\ell_1}^{\ell_2,\ell_3}$ have the following degrees of freedom vanished,   we   have
\a{    &\quad \ \partial_x p_{k}(\frac{j\b x_2+(k-3-j)\b x_3}{k-3})\\
        & = 1\cdot\frac{j^2}{(k-3)^2}\cdot\frac 1 h
      p_{k-5}(\frac{j\b x_2+(k-3-j)\b x_3}{k-3})\\
        & =0, \; \qquad
   j=1,\dots,k-4,  }
and consequently $p_{k-5}|_{\b x_2\b x_3}=0$.
  
Thus, factoring out the factor again, we have
\a{ p_k= \lambda_{14}^2 \lambda_{43}^2 \lambda_{23}^2 p_{k-6}
    \quad \ \t{for some } \ p_{k-6}\in P_{k-6}(T). }
Evaluating the function-value degrees of freedom on edge $\b x_1\b x_4$
    (one more than the $y$-derivative degrees of derivative), cf. Figure \ref{T-k}, we get
\a{  &\quad \  p_k(\frac{j\b x_1+(k-2-j)\b x_2}{k-2})\\
    &= 1^2\cdot  \frac {j^2}{(k-2)^2}\cdot \frac{(k-2-j)^2}{(k-2)^2} 
          \cdot p_{k-6}(\frac{j\b x_3+(k-2-j)\b x_2}{k-2})\\
   &=0,\qquad\qquad j=3,\dots,k-3,
  }and $p_{k-6}|_{\b x_1\b x_2}=0$. Thus,
\a{ p_k= \lambda_{14}^2 \lambda_{43}^2 \lambda_{23}^2 \lambda_{12} p_{k-7}
    \quad \ \t{for some } \ p_{k-7}\in P_{k-7}(T). } 
If $k=6$, we would have $p_k=0$ above.
Evaluating the $y$-derivative degrees of freedom on $\b x_1\b x_2$, cf. Figure \ref{T-k}, we get
\a{  &\quad \  \partial_y p_k(\frac{j\b x_1+(k-3-j)\b x_2}{k-3})\\
    &= \frac {j^2}{(k-3)^2} \cdot  \frac {(k-3-j)^2}{(k-3)^2}\cdot\frac 1 h\cdot
          \cdot p_{k-7}(\frac{j\b x_1+(k-3-j)\b x_2}{k-3})\\
   &=0,\qquad\qquad j=3,\dots,k-4,
  }and $p_{k-7}|_{\b x_1\b x_2}=0$. It leads to
\a{ p_k= \lambda_{14}^2 \lambda_{43}^2 \lambda_{23}^2 \lambda_{12}^2 p_{k-8}
    \quad \ \t{for some } \ p_{k-8}\in P_{k-8}(T). } 
Because the four factors are positive at the $\dim P_{k-8}$ internal Lagrange nodes in
  the last line of degrees of freedom \eqref{d-k}, and these $\dim P_{k-8}$ internal Lagrange nodes 
   are also the degrees of freedom of $b_{\ell_1}^{\ell_2,\ell_3}$ in \eqref{dof2},
   they force $p_{k-8}=0$ at these points and thus, $p_{k-8}$ itself is zero.

Evaluating $p$ in \eqref{p-k} sequentially at the degrees of freedom of $b_{\ell_1}^{\ell_2,\ell_3}$,
  it follows that
  \a{ c_1=\dots=c_{8} = 0, \quad \t{and }\ p=0.  }
The proof is complete.
\end{proof}

\section{The finite element solution and convergence }

The global bubble-enriched $C^1$-$P_k$ finite element space is defined by, for all $k\ge 4$, 
\an{\label{V-h} V_h=\{v_h\in H^2_0(\Omega) : v_h|_T \in V_k(T) \quad \forall T\in \mathcal Q_h \},
   } where $V_k(T)$ is defined in \eqref{V-4}, or \eqref{V-5}, or \eqref{V-k}.

 The finite element discretization of the biharmonic equation \eqref{bi} reads:
   Find $u\in V_h$ such that
\an{\label{finite} (\Delta u, \Delta v) = (f, v) \quad \forall v\in V_h, }
where $V_h$ is defined in \eqref{V-h}.
            
\begin{lemma}  The finite element problem \eqref{finite} has a unique solution.
\end{lemma}
                        
\begin{proof}  As \eqref{finite} is a square system of finite linear equations,
  we only need to prove the uniqueness.
Let $f=0$ and $v_h=u_h$ in \eqref{finite}.
It follows $\Delta u_h = 0 $ on the domain.  Let $v\in H^2_0(\Omega)$ be the solution
  of \eqref{bi} with $f =\Delta u_h$, as $u_h\in H^2_0(\Omega)$.
Because $u_h\in C^1(\Omega)$,  we have
\a{ 0 &= \int_{\Omega} \Delta u_h v d\b x 
     = \int_{\Omega} -\nabla u_h \nabla v d\b x  
      =\int_{\Omega} (u_h)^2 d\b x.  }
Thus, $u_h=0$. The proof is complete.
\end{proof}

For convergence, the analysis is standard,  as we have $C^1$ conforming finite elements.

\begin{theorem}  Let $ u\in H^{k+1}\cap H^2_0(\Omega)$ be
    the exact solution of the biharmoic equation \eqref{bi}.  
   Let $u_h$ be the $C^1$-$P_k$ finite element solution of \eqref{finite}.   
   Assuming the full-regularity on \eqref{bi}, it holds 
  \a{  \| u- u_h\|_{0} + h^2  |  u- u_h |_{2}  
         & \le Ch^{k+1} | u|_{k+1}, \quad k\ge 6.  } 
\end{theorem}
                        
\begin{proof} As $V_h\subset H^2_0(\Omega)$,  from \eqref{bi} and \eqref{finite},  we get
\a{ (\Delta ( u- u_h), \Delta v_h)=0\quad \forall v_h\in   V_{h }. }
Applying the Schwartz inequality,  we get 
\a{    |   u- u_h|_{2}^2  
     & = C (\Delta(  u-  u_h), \Delta(  u- u_h ))\\ &= C (\Delta(  u-  u_h), \Delta(  u- I_h u))\\ 
     &\le C |   u- u_h|_{2} |   u- I_h u |_{2} \\ 
     & \le Ch^{k-1} |u|_{k+1} |   u- u_h|_{2}  ,} 
      where $ I_h  u$ is the nodal interpolation defined by DOFs in \eqref{d-4} or  \eqref{d-5} or
         \eqref{d-k}.
As $V_k(T)\supset P_k(T)$,  we have $I_h u|_T = u|_T$ if $u\in P_k(T)$, 
  i.e., $I_h$ preserves $P_k$ functions locally.
  Such an interpolation operator is
   $H^2$ stable and consequently of the
    optimal order of convergence, by modifying the standard theory in \cite{Girault,Scott-Zhang}.

For the $L^2$ convergence,  we need an $H^4$ regularity for the dual problem: Find 
  $w\in H^2_0(\Omega)$ such that
  \an{ \label{d2}
    (\Delta w, \Delta v) &=(u-u_h, v), \ \forall v \in H^2_0(\Omega), }
  where \a{ |w|_4 \le C \|u-u_h\|_0 . }
Thus, by \eqref{d2}, \a{ \|u-u_h\|_0^2 &=(\Delta w, \Delta (u-u_h) ) = 
(\Delta (w-w_h), \Delta (u-u_h) ) \\
  & \le C h^2 |w|_{4}  h^{k-1} | u | _{k+1} 
  \\&  \le C h^{ k+1 }   | u | _{k+1} \|u-u_h\|_0. }
                 The proof is complete.
\end{proof}

\section{Numerical Experiments}

In the numerical computation,  we solve the biharmonic equation \eqref{bi}
   on the unit square domain $\Omega=(0,1)\times(0,1)$. 
We choose an $f$ in \eqref{bi} so that the exact solution is
\an{\label{s2}
   u =\sin^2(\pi x)\sin^2(\pi y).  }  
   
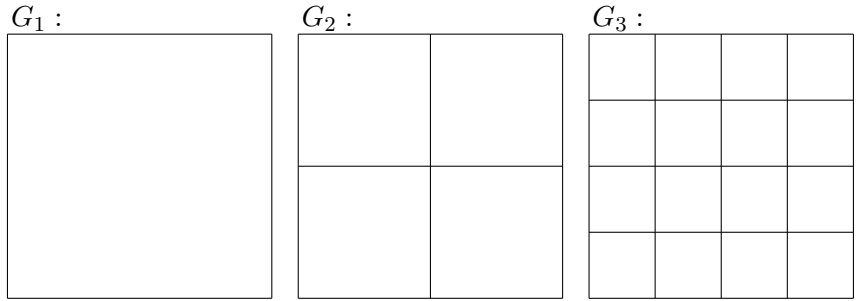
\begin{figure}[H]
\begin{center}\setlength\unitlength{1.0pt}\centering
\begin{picture}(330,115)(0,0) \put(5,101){$G_1:$}  \put(115,101){$G_2:$} \put(225,101){$G_3:$} 
\put(0,-2){ \begin{picture}(100,100)(0,0) \multiput(0,0)(100,0){2}{\line(0,1){100}}
       \multiput(0,0)(0,100){2}{\line(1,0){100}} \end{picture} }
\put(110,-2){ \begin{picture}(100,100)(0,0) \multiput(0,0)(50,0){3}{\line(0,1){100}}
       \multiput(0,0)(0,50){3}{\line(1,0){100}} \end{picture} }
\put(220,-2){ \begin{picture}(100,100)(0,0) \multiput(0,0)(25,0){5}{\line(0,1){100}}
       \multiput(0,0)(0,25){5}{\line(1,0){100}} \end{picture} }
\end{picture}\end{center}
\caption{The first three square grids for computing  \eqref{s2} in Tables \ref{t1}--\ref{t5}. }
\label{f-21}
\end{figure}

 We compute the solution \eqref{s2} on the square grids shown in Figure \ref{f-21}, by 
  the newly constructed $C^1$-$P_k$, $k=4,5,6,7,8$, finite elements \eqref{V-h}.
The results are listed in Tables \ref{t1}--\ref{t5}, where we can see that the optimal orders of convergence 
  are achieved in all cases.  
Additionally, we computed the corresponding $C^1$-$Q_k$ BFS finite element solutions in
  these tables.
The two solutions are about equally good.
The number of unknowns for the $C^1$-$P_4$ element is about 2/3 of that for the $C^1$-$Q_4$ element.
But the $C^1$-$P_k$ finite elements would have about
   $1/2$ of unknowns comparing to the $C^1$-$Q_k$ elements, eventually.
In the last row of some tables, the computer accuracy is reached, i.e., the round-off error is
  more than the  truncation error.

\begin{table}[H]
  \centering  \renewcommand{\arraystretch}{1.1}
  \caption{Error profile on the square meshes shown as in Figure \ref{f-21}, 
     for computing \eqref{s2}. }
  \label{t1}
\begin{tabular}{c|cc|cc|r}
\hline
grid & \multicolumn{2}{c|}{ $\| u-u_h\|_{0}$  \; $O(h^r)$}  
  &  \multicolumn{2}{c}{  $|u-u_h|_{2}$ \;$O(h^r)$} & $\dim V_h$  \\ \hline
    &  \multicolumn{5}{c}{ By the $C^1$-$Q_4$ BFS element. }   \\
\hline   
 1&    0.837E-01 &  0.0&    0.287E+01 &  0.0 &       25\\
 2&    0.939E-02 &  3.2&    0.161E+01 &  0.8 &       64\\
 3&    0.150E-03 &  6.0&    0.147E+00 &  3.5 &      196\\
 4&    0.461E-05 &  5.0&    0.184E-01 &  3.0 &      676\\
 5&    0.143E-06 &  5.0&    0.231E-02 &  3.0 &     2500\\
 6&    0.447E-08 &  5.0&    0.288E-03 &  3.0 &     9604\\
 7&    0.162E-09 &  4.8&    0.360E-04 &  3.0 &    37636\\
\hline 
    &  \multicolumn{5}{c}{ By the $C^1$-$P_4$ serendipity element \eqref{V-h}. }   \\
\hline   
 1&    0.375E+00 &  0.0&    0.174E+02 &  0.0 &       20\\
 2&    0.938E-02 &  5.3&    0.161E+01 &  3.4 &       48\\
 3&    0.128E-02 &  2.9&    0.533E+00 &  1.6 &      140\\
 4&    0.307E-04 &  5.4&    0.735E-01 &  2.9 &      468\\
 5&    0.871E-06 &  5.1&    0.992E-02 &  2.9 &     1700\\
 6&    0.264E-07 &  5.0&    0.127E-02 &  3.0 &     6468\\
 7&    0.828E-09 &  5.0&    0.159E-03 &  3.0 &    25220\\
\hline 
    \end{tabular}%
\end{table}%

\begin{table}[H]
  \centering  \renewcommand{\arraystretch}{1.1}
  \caption{Error profile on the square meshes shown as in Figure \ref{f-21}, 
     for computing \eqref{s2}. }
  \label{t2}
\begin{tabular}{c|cc|cc|r}
\hline
grid & \multicolumn{2}{c|}{ $\| u-u_h\|_{0}$  \; $O(h^r)$}  
  &  \multicolumn{2}{c}{  $|u-u_h|_{2}$ \;$O(h^r)$} & $\dim V_h$  \\
   \hline
    &  \multicolumn{5}{c}{ By the $C^1$-$Q_5$ BFS element. }   \\
\hline   
 1&    0.324E-01 &  0.0&    0.435E+01 &  0.0 &       36\\
 2&    0.138E-03 &  7.9&    0.918E-01 &  5.6 &      100\\
 3&    0.789E-05 &  4.1&    0.146E-01 &  2.7 &      324\\
 4&    0.130E-06 &  5.9&    0.912E-03 &  4.0 &     1156\\
 5&    0.206E-08 &  6.0&    0.570E-04 &  4.0 &     4356\\
 6&    0.302E-10 &  6.1&    0.356E-05 &  4.0 &    16900\\
\hline 
    &  \multicolumn{5}{c}{ By the $C^1$-$P_5$ serendipity element \eqref{V-h}. }   \\
\hline   
 1&    0.375E+00 &  0.0&    0.136E+02 &  0.0 &       28\\
 2&    0.486E-01 &  2.9&    0.550E+01 &  1.3 &       72\\
 3&    0.698E-03 &  6.1&    0.194E+00 &  4.8 &      220\\
 4&    0.109E-04 &  6.0&    0.988E-02 &  4.3 &      756\\
 5&    0.175E-06 &  6.0&    0.567E-03 &  4.1 &     2788\\
 6&    0.275E-08 &  6.0&    0.342E-04 &  4.0 &    10692\\ 
\hline 
    \end{tabular}%
\end{table}%

\begin{table}[H]
  \centering  \renewcommand{\arraystretch}{1.1}
  \caption{Error profile on the square meshes shown as in Figure \ref{f-21}, 
     for computing \eqref{s2}. }
  \label{t3}
\begin{tabular}{c|cc|cc|r}
\hline
grid & \multicolumn{2}{c|}{ $\| u-u_h\|_{0}$  \; $O(h^r)$}  
  &  \multicolumn{2}{c}{  $|u-u_h|_{2}$ \;$O(h^r)$} & $\dim V_h$  \\
   \hline
    &  \multicolumn{5}{c}{ By the $C^1$-$Q_6$ BFS element. }   \\
\hline   
 1&    0.157E-02 &  0.0&    0.802E+00 &  0.0 &       49\\
 2&    0.706E-04 &  4.5&    0.499E-01 &  4.0 &      144\\
 3&    0.394E-06 &  7.5&    0.115E-02 &  5.4 &      484\\
 4&    0.310E-08 &  7.0&    0.360E-04 &  5.0 &     1764\\
 5&    0.258E-10 &  6.9&    0.113E-05 &  5.0 &     6724\\
\hline 
    &  \multicolumn{5}{c}{ By the $C^1$-$P_6$ serendipity element \eqref{V-h}. }   \\
\hline   
 1&    0.375E+00 &  0.0&    0.137E+02 &  0.0 &       36\\
 2&    0.313E-02 &  6.9&    0.498E+00 &  4.8 &       96\\
 3&    0.408E-04 &  6.3&    0.245E-01 &  4.3 &      300\\
 4&    0.221E-06 &  7.5&    0.703E-03 &  5.1 &     1044\\
 5&    0.138E-08 &  7.3&    0.209E-04 &  5.1 &     3876\\ 
\hline 
    \end{tabular}%
\end{table}%

\begin{table}[H]
  \centering  \renewcommand{\arraystretch}{1.1}
  \caption{Error profile on the square meshes shown as in Figure \ref{f-21}, 
     for computing \eqref{s2}. }
  \label{t4}
\begin{tabular}{c|cc|cc|r}
\hline
grid & \multicolumn{2}{c|}{ $\| u-u_h\|_{0}$  \; $O(h^r)$}  
  &  \multicolumn{2}{c}{  $|u-u_h|_{2}$ \;$O(h^r)$} & $\dim V_h$  \\
   \hline
    &  \multicolumn{5}{c}{ By the $C^1$-$Q_7$ BFS element. }   \\
\hline   
 1&    0.115E-02 &  0.0&    0.379E+00 &  0.0 &       64\\
 2&    0.964E-06 & 10.2&    0.253E-02 &  7.2 &      196\\
 3&    0.183E-07 &  5.7&    0.763E-04 &  5.0 &      676\\
 4&    0.731E-10 &  8.0&    0.119E-05 &  6.0 &     2500\\ 
 5&    0.158E-10 &  2.2&    0.185E-07 &  6.0 &     9604\\
\hline 
    &  \multicolumn{5}{c}{ By the $C^1$-$P_7$ serendipity element \eqref{V-h}. }   \\
\hline   
 1&    0.375E+00 &  0.0&    0.140E+02 &  0.0 &       44\\
 2&    0.128E-02 &  8.2&    0.506E+00 &  4.8 &      120\\
 3&    0.679E-05 &  7.6&    0.409E-02 &  7.0 &      380\\
 4&    0.285E-07 &  7.9&    0.614E-04 &  6.1 &     1332\\
 5&    0.209E-09 &  7.1&    0.994E-06 &  6.0 &     4964\\
\hline 
    \end{tabular}%
\end{table}%

\begin{table}[H]
  \centering  \renewcommand{\arraystretch}{1.1}
  \caption{Error profile on the square meshes shown as in Figure \ref{f-21}, 
     for computing \eqref{s2}. }
  \label{t5}
\begin{tabular}{c|cc|cc|r}
\hline
grid & \multicolumn{2}{c|}{ $\| u-u_h\|_{0}$  \; $O(h^r)$}  
  &  \multicolumn{2}{c}{  $|u-u_h|_{2}$ \;$O(h^r)$} & $\dim V_h$  \\
   \hline
    &  \multicolumn{5}{c}{ By the $C^1$-$Q_8$ BFS element. }   \\
\hline   
 1&    0.531E-04 &  0.0&    0.716E-01 &  0.0 &       81\\
 2&    0.546E-06 &  6.6&    0.743E-03 &  6.6 &      256\\
 3&    0.755E-09 &  9.5&    0.433E-05 &  7.4 &      900\\
 4&    0.557E-11 &  7.1&    0.334E-07 &  7.0 &     3364\\
\hline 
    &  \multicolumn{5}{c}{ By the bubble-enriched $C^1$-$P_8$ element \eqref{V-h}. }   \\
\hline   
 1&    0.465E-01 &  0.0&    0.389E+01 &  0.0 &       53\\
 2&    0.782E-04 &  9.2&    0.246E-01 &  7.3 &      148\\
 3&    0.567E-06 &  7.1&    0.425E-03 &  5.9 &      476\\
 4&    0.109E-08 &  9.0&    0.350E-05 &  6.9 &     1684\\
\hline 
    \end{tabular}%
\end{table}%

\end{document}